\documentclass[a4paper,12pt]{article}
\usepackage[top=2.5cm,bottom=2.5cm,left=2.5cm,right=2.5cm]{geometry}
\usepackage{cite, amsmath, amssymb, cases}
\numberwithin{equation}{section}
\usepackage{amssymb}
\usepackage{graphicx}
\usepackage{amsmath,amsthm,amssymb,lineno}
\usepackage{latexsym}
\usepackage{graphicx,booktabs,multirow}
\usepackage{tikz}
\usetikzlibrary{decorations.pathreplacing}
\usetikzlibrary{intersections}
\usepackage{enumerate}
\usepackage{graphicx,booktabs,multirow}
\usepackage{appendix}
\newtheorem{theorem}{Theorem}[section]

\newtheorem{lemma}[theorem]{Lemma}

\newtheorem{definition}[theorem]{Definition}
\usepackage[section]{placeins}
\def\NAT@def@citea{\def\@citea{\NAT@separator}}
\allowdisplaybreaks
\usepackage{enumitem}
\usepackage{amsmath}
\usepackage{caption}
\begin{document}
\vspace*{10mm}

\noindent
{\Large \bf Transversal and Hamiltonicity in a bipartite graph collection}

\vspace*{7mm}

\noindent
{\large \bf  Menghan Ma$^1$, Lihua You$^2$, Xiaoxue Zhang$^{2,*}$}
\noindent

\vspace{7mm}

\noindent
$^1$ Center for Discrete Mathematics, Fuzhou University, Fuzhou, 350108, P. R. China,
e-mail: {\tt menghanma664@163.com}\\
$^2$ School of Mathematical Sciences, South China Normal University, Guangzhou, 510631, P. R. China,
e-mail: {\tt ylhua@scnu.edu.cn}, {\tt zhang\_xx1209@163.com}.\\[2mm]
$^*$ Corresponding author
\noindent

\vspace{7mm}

\noindent
{\bf Abstract} \
Let $\mathbf{G}=\{G_1,\dots,G_{s}\}$ be a collection of $s$ bipartite graphs with the same bipartition $V=(X,Y)$. For a path $P$ with $V(P)=V$ and $|E(P)|=s$, if there exists an injection $\phi$: $E(P)\rightarrow [s]$ such that $e\in E(G_{\phi(e)})$ for each $e\in E(P)$, then we say that the Hamiltonian path $P$ is a $\mathbf{G}$-transversal. A bipartite graph collection  $\mathbf{G}$ is called Hamiltonian connected if for any two vertices $x\in X$ and $y\in Y$, there exists a $\mathbf{G}$-transversal isomorphic to a Hamiltonian path between $x$ and $y$.
In this paper, we give the minimum degree conditions that ensure the existence of a $\mathbf{G}$-transversal isomorphic to a Hamiltonian path and the Hamiltonian connectivity of a balanced bipartite graph collection $\mathbf{G}$, which improve the results of [Hu, Li, Li and Xu, Discrete Math., 2024]. Moreover, we also provide a minimum degree condition that guarantees a nearly balanced bipartite graph collection $\mathbf{G}$ contains a $\mathbf{G}$-transversal isomorphic to a Hamiltonian path.
\vspace{5mm}



\baselineskip=0.30in

\section{Introduction}

\hspace{1.5em}In this paper, we only consider finite and simple graphs. Undefined terminology and notation can be found in \cite{Bondy}.
Let $G=(V(G),E(G))$ be a simple graph with vertex set $V(G)$ and edge set $E(G)$.
For a vertex $u\in V(G)$, the set of neighbours of a vertex $u$ in $G$ is denoted by $N_{G}(u)$, and $d_{G}(u)=|N_{G}(u)|$ is the \emph{degree} of $u$ in $G$.
Let $\delta(G)=\min\{d_{G}(u): u\in V(G)\}$ denote the \emph{minimum degree} of $G$.
A path (cycle) which contains every vertex of a graph is called a \emph{Hamiltonian path (cycle)}. For positive integers $a<b$, let $[a]=\{1,2,\dots,a\}$ and $[a,b]=\{a,a+1,\dots,b\}$.

The concept of transversals frequently arises in combinatorics. Recently, the study on transversal in a graph collection has attracted extensive attention. This concept was first proposed by Joos and Kim \cite{Joos} in $2020$. Let $\mathbf{G}=\{G_1,\dots,G_s\}$ be a collection of $s$ graphs (not necessarily distinct) on the same vertex set $V$.
For a graph $H$ with $V(H)\subseteq V$ and $|E(H)|\leq s$, if there exists an injection $\phi$: $E(H)\rightarrow [s]$ such that $e\in E(G_{\phi(e)})$ for each $e\in E(H)$, then we say that $H$ is a \emph{partial $\mathbf{G}$-transversal} if $|E(H)|<s$, $H$ is a \emph{$\mathbf{G}$-transversal} if $|E(H)|=s$. We further say that $\mathbf{G}$ contains a (partial) transversal $H$.

A natural question arises: what properties should be imposed on $\mathbf{G}$ to guarantee that there exists a (partial) $\mathbf{G}$-transversal isomorphic to $H$?
In fact, some scholars have already begun to investigate this topic and have achieved some remarkable research results. Here we introduce two classic results.

In $1952$, Dirac \cite{Dirac} proved that every $n$-vertex graph with minimum degree at least $\frac{n}{2}$ contains a Hamiltonian cycle, a result now widely known as Dirac's theorem. In $2020$, Joos and Kim \cite{Joos} proved that if $\mathbf{G}=\{G_1,\dots, G_n\}$ is a collection of $n$ graphs on the same vertex set $V$ with $|V|=n$ such that $\delta(G_i)\geq\frac{n}{2}$ for each $i\in[n]$, then there exists a $\mathbf{G}$-transversal isomorphic to a Hamiltonian cycle. This result can be regarded as a generalization of Dirac's theorem, since it reduces to Dirac's theorem when $G_1=\dots=G_n$.

In $1907$, Mantel \cite{Mantel} showed that an $n$-vertex graph with more than $\frac{n^2}{4}$ edges has a triangle. In $2020$, Aharoni, DeVos, Gonz\'{a}lez Hermosillo de la Maza and \v{S}\'{a}mal in \cite{Aharoni} proved that if $\mathbf{G}=\{G_1,G_2,G_3\}$ is a collection of three graphs on the same vertex set $V$ with $|V|=n$ such that $|E(G_i)|>\frac{1+\tau^2}{4}n^2\approx 0.2557n^2$ for each $i\in[3]$, then there exists a $\mathbf{G}$-transversal isomorphic to a triangle, where $\tau=\frac{4-\sqrt{7}}{9}$. Moreover, they also show that this bound is best possible.
This result extends Mantel's theorem from a single graph to a collection of graphs.
It is worth noting that the two results require different bounds on the number of edges.

The above results extend the classic theorems of graph theory to the framework of a general graph collection, and more results are available in \cite{Li2023,Li2024,Cheng,Liu,Sun}. In addition, there have also been studies on a bipartite graph collection \cite{Bradshaw,Hu}, a directed graph collection \cite{Li2025,Gerbner,Babinski,Cheng2023,Chakraborti} and a hypergraph collection \cite{Cheng2023,Cheng2025,Lu2021,Lu2022,Gao2022}. For more results, readers may refer to a survey \cite{SunandWang}.

In this paper, we focus on the study of transversal in a bipartite graph collection. A bipartite graph with the bipartition $V=(X,Y)$ is called \emph{balanced} if $|X|=|Y|$.
In $2020$, Bradshaw \cite{{Bradshaw}} considered the minimum degree condition for a balanced bipartite graph collection $\mathbf{G}$ to guarantee that there exists a (partial) $\mathbf{G}$-transversal isomorphic to a cycle of order $l$ for each even intenger $l\in[4,2n]$. In $2024$, Hu, Li, Li and Xu \cite{Hu} provided the minimum degree condition for a balanced bipartite graph collection $\mathbf{G}$ on vertex set $V$ such that each vertex $v\in V$ is contained in a (partial) $\mathbf{G}$-transversal isomorphic to a cycle of order $l$ for each even intenger $l\in[4,2n]$, and gave the following result on Hamiltonian connectivity of a bipartite graph collection, which is stated as Theorem \ref{thm1-2}. For this purpose, they first proved the following Theorem \ref{thm1-1} in \cite{Hu}.

\begin{theorem}[\rm\!\!\cite{Hu}]\label{thm1-1}
Let $\mathbf{G}=\{G_1,\dots,G_{2n}\}$ be a collection of $2n$ balanced bipartite graphs of order $2n$ with the same bipartition $V=(X,Y)$. If $\delta(G_i)\geq\left \lceil \frac{n}{2}  \right \rceil$ for each $i\in[2n]$, then one of the following statements holds:
\begin{enumerate}[label=(\roman*), font=\normalfont, align=left, leftmargin=1em]
\item $\mathbf{G}$ contains a partial transversal isomorphic to a Hamiltonian path;
\item $n$ is even and $G_1=\dots=G_{2n}=K_{\frac{n}{2},\frac{n}{2}}\cup K_{\frac{n}{2},\frac{n}{2}}$.
\end{enumerate}
\end{theorem}

\begin{center}
\scalebox{1}[1]{\includegraphics{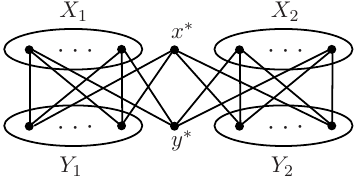}}\\
\captionof{figure}{Graph $F$.}
\end{center}

To facilitate the description, we first introduce a graph, denoted by $F$, as shown in Figure 1. Let $n$ be an odd integer, $F$ be a balanced bipartite graph with vertex bipartition $V(F)=(X,Y)$, where $X=\{x^*\}\cup X_1\cup X_2$ with $|X_1|=|X_2|=\frac{n-1}{2}$ and $Y=\{y^*\}\cup Y_1\cup Y_2$ with $|Y_1|=|Y_2|=\frac{n-1}{2}$, $E(F)$ be composed of all edges between $\{x^*\}$ and $Y_1\cup Y_2$, all edges between $\{y^*\}$ and $X_1\cup X_2$, and all edges between $X_i$ and $Y_i$ for $i=1,2$. Now we define the following graph collection: let $m$ be a positive integer and $\mathbf{F_m}=\{F_1,\dots,F_{m}\}$ be a bipartite graph collection with a common vertex bipartition such that $F_i=F$ or $F_i=F\cup\{x^*y^*\}$ for each $i\in[m]$.

A bipartite graph collection $\mathbf{G}$ with the same bipartition $V=(X,Y)$ is said to be \emph{Hamiltonian connected} if for any two vertices $x\in X$ and $y\in Y$, there exists a (partial) $\mathbf{G}$-transversal isomorphic to a Hamiltonian path between $x$ and $y$.

\begin{theorem}[\rm\!\!\cite{Hu}]\label{thm1-2}
Let $\mathbf{G}=\{G_1,\dots,G_{2n}\}$ be a collection of $2n$ balanced bipartite graphs of order $2n$ with the same bipartition $V=(X,Y)$. If $\delta(G_i)\geq\left \lceil \frac{n+1}{2}  \right \rceil$ for each $i\in[2n]$, then $\mathbf{G}$ is Hamiltonian connected or $\mathbf{G}=\mathbf{F_{2n}}$.
\end{theorem}

Note that Theorems \ref{thm1-1} and \ref{thm1-2} deal with a collection of $2n$ balanced bipartite graphs of order $2n$, rather than $2n-1$ balanced bipartite graphs. As noted by the authors in \cite{Hu}, it would be interesting to study the $\mathbf{G}$-transversal isomorphic to a Hamiltonian path in a collection $\mathbf{G}$ of $2n-1$ balanced bipartite graphs under the same degree condition. This motivates us to investigate this problem. In this paper, we obtain the following two theorems, which improve upon Theorems \ref{thm1-1} and \ref{thm1-2}, respectively.

\begin{theorem}\label{thm1-3}
Let $\mathbf{G}=\{G_1,\dots,G_{2n-1}\}$ be a collection of $2n-1$ balanced bipartite graphs of order $2n$ with the same bipartition $V=(X,Y)$. If $\delta(G_i)\geq\left \lceil \frac{n}{2}  \right \rceil$ for each $i\in[2n-1]$, then one of the following statements holds:
\begin{enumerate}[label=(\roman*), font=\normalfont, align=left, leftmargin=1em]
\item $\mathbf{G}$ contains a transversal isomorphic to a Hamiltonian path;
\item $n$ is even and $G_1=\dots=G_{2n-1}=K_{\frac{n}{2},\frac{n}{2}}\cup K_{\frac{n}{2},\frac{n}{2}}$.
\end{enumerate}
\end{theorem}

\begin{theorem}\label{thm1-4}
Let $\mathbf{G}=\{G_1,\dots,G_{2n-1}\}$ be a collection of $2n-1$ balanced bipartite graphs of order $2n$ with the same bipartition $V=(X,Y)$. If $\delta(G_i)\geq\left \lceil \frac{n+1}{2}  \right \rceil$ for each $i\in[2n-1]$, then $\mathbf{G}$ is Hamiltonian connected or $\mathbf{G}=\mathbf{F_{2n-1}}$.
\end{theorem}

A bipartite graph with the bipartition $V=(X,Y)$ is called \emph{nearly balanced} if $|X|-|Y|\in \{-1,1\}$. We give a minimum degree condition for a nearly balanced bipartite graph collection to contain a transversal isomorphic to a Hamiltonian path.

\begin{theorem}\label{thm1-5}
Let $\mathbf{G}=\{G_1,\dots,G_{2n}\}$ be a collection of $2n$ nearly balanced bipartite graphs of order $2n+1$ with the same bipartition $V=(X,Y)$. If $\delta(G_i)\geq\left \lceil \frac{n+1}{2}  \right \rceil$ for each $i\in[2n]$, then $\mathbf{G}$ contains a transversal isomorphic to a Hamiltonian path.
\end{theorem}

It is worth mentioning that $\left \lceil \frac{n+1}{2}  \right \rceil$ is best possible. Now we give an example to illustrate this. Let $n$ be even, $\mathbf{G}=\{G_1,\dots,G_{2n}\}$ be a collection of $2n$ nearly balanced bipartite graphs with the same bipartition $V=(X,Y)$, and $G_1=\dots=G_{2n}=K_{\frac{n}{2},\frac{n}{2}}\cup K_{\frac{n}{2}+1,\frac{n}{2}}$. Clearly, $\delta(G_i)\geq\left \lceil \frac{n}{2}  \right \rceil$ for each $i\in[2n]$, but $\mathbf{G}$ contains no transversal isomorphic to a Hamiltonian path.

The rest of this paper is organized as follows. In Section $2$, we introduce the necessary notation as well as some results that will be used in subsequent proofs.
In Sections $3$-$5$, we present the proofs of Theorems \ref{thm1-3}-\ref{thm1-5}, respectively.

\section{Preliminaries}
\hspace{1.5em}In this section, we present some notation and results that will be used later.

\subsection{Notation}
\hspace{1.5em} For two disjoint vertex subsets $U$ and $W$ of $G$, $G[U,W]$ denotes the bipartite subgraph of $G$ with vertex bipartition $(U,W)$ and edge set $E(G[U,W])=\{uw\in E(G): u\in U,w\in W$\}. We use $E_G[U,W]$ for $E(G[U,W])$ for short.
For a path $P=u_1u_2\cdots u_k$, we denote the subpath of $P$ between $u_i$ and $u_j$ by $u_iPu_j$, where $i,j\in[k]$. For two distinct vertices $v_i,v_j$ in a cycle $C=v_1v_2\cdots v_kv_1$ with $i<j$, we denote the two paths between them as the forward segment $v_iCv_j=v_iv_{i+1}\cdots v_j$ and the backward segment $v_iC^{-}v_j=v_iv_{i-1}\cdots v_j$, where subscripts are taken modulo $k$. Let $im(\phi)$ be the image of $\phi$.

Joos and Kim introduced the auxiliary digraph technique in \cite{Joos}, and our proofs are also based on this important tool.
For a digraph $D$, let $A(D)$ be the arc set of $D$. For a vertex $u\in V(D)$, we denote the in-neighbourhood and out-neighbourhood of $u$ in $D$ by $N^{-}_{D}(u)$ and $N^{+}_{D}(x)$, and $d^{-}_{D}(u)=|N^{-}_{D}(u)|$ and $d^{+}_{D}(x)=|N^{+}_{D}(u)|$ be the \emph{indegree} and \emph{outdegree} of $u$, respectively.


\subsection{Lemmas}

\hspace{1.5em} We follow a method of Bradshaw \cite{Bradshaw} used for investigating the existence of a $\mathbf{G}$-transversal isomorphic to a Hamiltonian cycle in a bipartite graph collection $\mathbf{G}$, and obtain the following lemmas.

To proceed with the lemmas, we first define a digraph $D$ as follows.
\begin{definition}\label{def1-1}
Let $n\geq3$, $\mathbf{G}=\{G_1,\dots,G_{2n-1}\}$ be a collection of $2n-1$ balanced bipartite graphs with the same bipartition $V=(X,Y)$ with $X=\{u_1,u_3,\dots,u_{2n-3},x\}$ and $Y=\{u_2,u_4,\dots,u_{2n-2},y\}$, $C=u_1u_2\cdots u_{2n-2}u_1$ be a partial $\mathbf{G}$-transversal isomorphic to a cycle of order $2n-2$ with an associated injective $\phi$: $E(C)\to [2n-1]$ such that $\phi(u_iu_{i+1})=i$ for $i\in[2n-3]$ and $\phi(u_{2n-2}u_1)=2n-2$. Let $D=(V,A)$ be a digraph on vertex set $V$ such that
\begin{equation}\tag{1}\label{eq.1}
A=\bigcup_{u_i\in X\setminus\{x\}}\{\overrightarrow{u_iu}:u\in Y, u_iu\in E(G_i) \text{ and } u\neq u_{i+1}\}.
\end{equation}
We call $D$ the associated digraph with $\mathbf{G}$.
\end{definition}

\begin{lemma}\label{lem2-1}
Let $\mathbf{G}=\{G_1,\dots,G_{2n-1}\}$, $V$, $X$, $Y$, $C$, $D$, $x$ and $y$ be defined as in Definition \ref{def1-1}. If $\delta(G_i)\geq\left \lceil \frac{n}{2}  \right \rceil$ for each $i\in[2n-1]$ and there exists $k\in\{2,4,\dots,2n-2\}$ such that $d^{-}_{D}(u_k)\geq \left \lfloor \frac{n}{2} \right \rfloor$ and $|(N_{G_{2n-1}}(x)\cup N_{G_k}(y))\cap V(C)|\geq n-1$, then there exists a $\mathbf{G}$-transversal isomorphic to a Hamiltonian path between $x$ and $y$.
\end{lemma}

\begin{proof}
By contradiction, we assume that there exists no $\mathbf{G}$-transversal isomorphic to a Hamiltonian path between $x$ and $y$.
Without loss of generality, let $d^{-}_{D}(u_2)\geq \left \lfloor \frac{n}{2} \right \rfloor$ and $|(N_{G_{2n-1}}(x)\cup N_{G_2}(y))\cap V(C)|\geq n-1$.

Since $|V(C)|=2n-2$ and $|(N_{G_{2n-1}}(x)\cup N_{G_2}(y))\cap V(C)|\geq n-1$ and $\delta(G_i)\geq\left \lceil \frac{n}{2}  \right \rceil\geq2$ for each $i\in[2n-1]$,
there exists $u_t\in Y\cap V(C)$ such that $t$ is even and $u_t\in N_{G_{2n-1}}(x)$ and $N_{G_2}(y)\cap \{u_{t+1},u_{t-1}\}\neq \emptyset$. Without loss of generality, we say $yu_{t+1}\in E(G_2)$.

Clearly, $u_2u_3\notin E(G_t)$. Otherwise, $yu_{t+1}Cu_tx$ is a $\mathbf{G}$-transversal isomorphic to a Hamiltonian path between $x$ and $y$ with an associated injective $\phi_1$ that arises from $\phi$ by setting $\phi_1(yu_{t+1})=2$, $\phi_1(u_2u_3)=t$ and $\phi_1(xu_t)=2n-1$, a contradiction. It follows that $t\neq2$ by $u_2u_3\in E(G_2)$.

Now we define the sets $S^{-}$ and $S_{t}$ as follows:
\[
S^{-}=\{s\in\{1,3,\dots,2n-3\}:u_s\in N^{-}_{D}(u_2)\},
\]
\[
S_{t}=\{s\in\{1,3,\dots,2n-3\}:u_{s+1}u_3\in E(G_t)\}.
\]
Clearly, $1\notin S^{-}$ by $u_1\notin N^{-}_{D}(u_2)$ and $1\notin S_{t}$ by $u_2u_3\notin E(G_t)$. Then $|S^{-}|=d^{-}_{D}(u_2)\geq\left \lfloor \frac{n}{2} \right \rfloor$ by \eqref{eq.1}, and $|S_{t}|=|N_{G_t}(u_3)\cap V(C)|\geq \left \lceil \frac{n}{2}  \right \rceil-1$ by $\delta(G_t)\geq\left \lceil \frac{n}{2}  \right \rceil$. We have $S^{-}\cap S_{t}\neq\emptyset$ since $|S^{-}|+|S_{t}|\geq n-1>n-2\geq|S^{-}\cup S_{t}|$ by $S^{-}\cup S_{t}\subseteq\{3,\dots,2n-3\}$.
Thus there exists $s\in\{3,\dots,2n-3\}$ such that $u_s\in N^{-}_{D}(u_2)$ and $u_{s+1}u_3\in E(G_t)$.
Note that $s\neq t$ since $s$ is odd and $t$ is even.
It follows that $(C-\{u_2u_3,u_su_{s+1},u_tu_{t+1}\})\cup\{xu_t,u_3u_{s+1},u_2u_s ,yu_{t+1}\}$ is a $\mathbf{G}$-transversal isomorphic to a Hamiltonian path between $x$ and $y$ (see Figure 2 (a)) with an associated injective $\phi_2$ that arises from $\phi$ by setting $\phi_2(xu_t)=2n-1$, $\phi_2(u_3u_{s+1})=t$, $\phi_2(u_2u_s)=s$ and $\phi_2(yu_{t+1})=2$, a contradiction.

This completes the proof of Lemma \ref{lem2-1}.
\end{proof}

\begin{lemma}\label{lem2-2}
Let $\mathbf{G}=\{G_1,\dots,G_{2n-1}\}$, $V$, $X$, $Y$, $C$, $D$, $x$ and $y$ be defined as in Definition \ref{def1-1}. If there exists $k\in\{2,4,\dots,2n-2\}$ such that $d^{-}_{D}(u_k)+|N_{G_k}(x)\cap V(C)|\geq n-1$ and $u_{k+1}y\in E(G_{2n-1})$, then there exists a $\mathbf{G}$-transversal isomorphic to a Hamiltonian path between $x$ and $y$.

\end{lemma}

\begin{proof}
By contradiction, we assume that there exists no $\mathbf{G}$-transversal isomorphic to a Hamiltonian path between $x$ and $y$.

Without loss of generality, we assume $k=2$.
Since $u_3y\in E(G_{2n-1})$, we have $xu_2\notin E(G_2)$. Otherwise, $xu_2C^{-}u_3y$ is a $\mathbf{G}$-transversal isomorphic to a Hamiltonian path between $x$ and $y$, a contradiction.

Now we define the sets $S^{-}$ and $S$ as follows:
\[
S^{-} =\{s\in\{1,3,\dots,2n-3\}:u_s\in N^{-}_{D}(u_2)\},
\]
\[
S =\{s\in\{1,3,\dots,2n-3\}:xu_{s+1}\in E(G_{2})\}.
\]
Clearly, $1\notin S^{-}$ by $u_1\notin N^{-}_{D}(u_2)$ and $1\notin S$ by $xu_2\notin E(G_2)$. Thus $|S^{-}\cup S|\leq n-2$ by $S^{-}\cup S\subseteq\{3,\dots,2n-3\}$. In addition, $|S^{-}|+|S|=d^{-}_{D}(u_2)+|N_{G_2}(x)\cap V(C)|\geq n-1$, which implies $S^{-}\cap S\neq\emptyset$. It follows that there exists $s\in\{3,\dots,2n-3\}$ such that $u_su_2\in E(G_{s})$ by \eqref{eq.1} and $xu_{s+1}\in E(G_{2})$. Therefore, $xu_{s+1}Cu_2u_sC^{-}u_3y$ is a $\mathbf{G}$-transversal isomorphic to a Hamiltonian path between $x$ and $y$ (see Figure 2 (b)) with an associated injective $\phi_3$ that arises from $\phi$ by setting $\phi_3(u_{s+1}x)=2$, $\phi_3(u_2u_s)=s$ and $\phi_3(u_3y)=2n-1$, a contradiction.

This completes the proof of Lemma \ref{lem2-2}.
\end{proof}

\begin{center}
\scalebox{1.4}[1.4]{\includegraphics{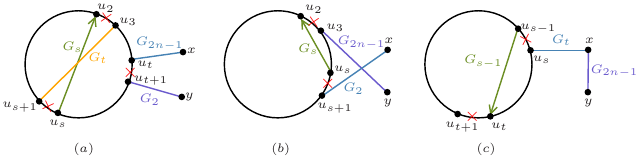}}\\
\captionof{figure}{$\mathbf{G}$-transversal isomorphic to a Hamiltonian path, where ``$\times$" represents the edge being deleted. $(a)$ illustrates the case where $3<t+1<s\leq2n-3$.}
\end{center}

\section{The proof of Theorem \ref{thm1-3}}
\hspace{1.5em}In this section, we prove Theorem \ref{thm1-3}. When $n=1$, the theorem is trivial. When $n=2$, it is straightforward to verify that the theorem holds. Therefore, we assume $n\geq3$ in the following proof.

\noindent\textbf{\textit{Proof of Theorem \ref{thm1-3}}.}
By contradiction, we assume neither $(\mathrm{i})$ nor $(\mathrm{ii})$ holds.

\noindent\textbf{Claim 3.1.}
$\mathbf{G}$ contains a partial transversal isomorphic to a path of order $2n-2$.

\begin{proof}
Let $P$ be a partial $\mathbf{G}$-transversal isomorphic to a path such that $|V(P)|$ is maximum.
If $|V(P)|\geq 2n-2$, then Claim 3.1 holds.

Now we assume that $P=u_1u_2\cdots u_p$ with $|V(P)|=p<2n-2$. Then $|V\setminus V(P)|\geq 3$. Let $X'=X\setminus (V(P)\cap X)$ and $Y'=Y\setminus (V(P)\cap Y)$. Clearly, $|X'|\geq1$ and $|Y'|\geq1$. Without loss of generality, let $P$ have an associated injective $\varphi$: $E(P)\to [2n-1]$ with $\varphi(u_iu_{i+1})=i$ for $i\in[p-1]$ and $\{m_1,m_2,m_3\}\subseteq[2n-1]\setminus im(\varphi)=[p,2n-1]$.

Without loss of generality, we assume $u_1\in X$. Pick $y'\in Y'$. Then we consider the following sets:
\[
S_1=\{s\in[p-1]:u_1u_{s+1}\in E(G_{m_1})\},
\]
\[
S_2=\{s\in[p-1]:y'u_s\in E(G_{m_2})\}.
\]

Firstly, we show that $S_1\cap S_2=\emptyset$. Otherwise, we take $s\in S_1\cap S_2$, then $P'=y'u_sPu_1u_{s+1}Pu_p$ is a partial $\mathbf{G}$-transversal isomorphic to a path that is longer than $P$, where $P'$ has an associated injective $\varphi_1$ that arises from $\varphi$ by setting $\varphi_1(y'u_s)=m_2$ and $\varphi_1(u_1u_{s+1})=m_1$, a contradiction.

Secondly, we show that $n$ is even and $G_{m_2}[X',Y']$ is a complete bipartite graph.
By the maximality of $|V(P)|$, we have $N_{G_{m_1}}(u_1)\subseteq V(P)$, it follows that $|S_1|=d_{G_{m_1}}(u_1)\geq\delta(G_{m_1})\geq\left \lceil \frac{n}{2}  \right \rceil$. Moreover, since $N_{G_{m_2}}(y')\subseteq X$ and $S_1\cap S_2=\emptyset$, we have $d_{G_{m_2}}(y')\leq n-|S_1|\leq \left \lfloor \frac{n}{2} \right \rfloor$, then $d_{G_{m_2}}(y')=\frac{n}{2}$ and $n$ is even by $d_{G_{m_2}}(y')\geq\delta(G_{m_2})\geq \left \lceil \frac{n}{2}  \right \rceil$, and thus $d_{G_{m_2}}(y')=n-|S_1|$ and $|S_1|=\frac{n}{2}$, which implies that $y'u\in E(G_{m_2})$ for any $u\in X'$, and $G_{m_2}[X',Y']$ is a complete bipartite graph since $y'$ is an arbitrary vertex in $Y'$.

Note that $N_{G_{m_1}}(u_1)\subseteq V(P)\cap Y$ and $N_{G_{m_3}}(u_p)\subseteq V(P)$ by the maximality of $|V(P)|$. Let $x'\in X'$. Then $x'y'\in E(G_{m_2})$. Next, we consider the following two cases.

$\mathbf{Case~1:}$ $u_p\in X$.

Clearly, $p$ is odd, and we can define $S_3$ and $S_4$ as follows:
\[
S_3=\{s\in\{0,2,\dots,p-3\}:u_1u_{s+2}\in E(G_{m_1})\},
\]
\[
S_4=\{s\in\{2,4,\dots,p-1\}:u_pu_s\in E(G_{m_3})\},
\]
and thus $|S_3|=d_{G_{m_1}}(u_1)\geq \delta(G_{m_1})\geq \frac{n}{2}$, $|S_4|=d_{G_{m_3}}(u_p)\geq \delta(G_{m_3})\geq \frac{n}{2}$.

Now we show $S_3\cap S_4=\emptyset$. Otherwise, assume $s\in S_3\cap S_4$, then $C=u_1Pu_su_pPu_{s+2}u_1$ is a partial $\mathbf{G}$-transversal isomorphic to a cycle, where $C$ has an associated injective $\varphi_2$ that arises from $\varphi$ by setting $\varphi_2(u_su_p)=m_3$ and $\varphi_2(u_1u_{s+2})=m_1$.
Clearly, $s\in [2n-1]\setminus im(\varphi_2)$ and $N_{G_s}(y')\cap V(P)=N_{G_s}(x')\cap V(P)=\emptyset$ by the maximality of $|V(P)|$ and $x'y'\in E(G_{m_2})$, which implies $N_{G_s}(y')\subseteq X'$ and $N_{G_s}(x')\subseteq Y'$, thus $|X'|\geq \frac{n}{2}$ and $|Y'|\geq \frac{n}{2}$ by $\delta(G_s)\geq \frac{n}{2}$. Since $N_{G_{m_1}}(u_1)\subseteq V(P)\cap Y$, we have $|\{u_2,u_4,\dots,u_{p-1}\}|=|V(P)\cap Y|\geq \frac{n}{2}$, thus $|\{u_1,u_3,\dots,u_{p}\}|=|V(P)\cap X|\geq \frac{n}{2}+1$, it follows that $|X|=|V(P)\cap X|+|X'|\geq n+1$, which contradicts the fact that $|X|=n$.

Clearly, $|S_3\cup S_4|\leq \frac{p+1}{2}<\frac{2n-1}{2}$. Since $S_3\cap S_4=\emptyset$, $|S_3|\geq \frac{n}{2}$ and $|S_4|\geq \frac{n}{2}$, we have $n\leq|S_3|+|S_4|=|S_3\cup S_4|<\frac{2n-1}{2}$, a contradiction.

$\mathbf{Case~2:}$ $u_p\in Y$.

Similar to Case $1$, we can define the sets $S'_3$ and $S'_4$ as follows:
\[
S'_3=\{s\in\{1,3,\dots,p-1\}:u_1u_{s+1}\in E(G_{m_1})\},
\]
\[
S'_4=\{s\in\{1,3,\dots,p-1\}:u_pu_s\in E(G_{m_3})\}.
\]
Clearly, $p$ is even, $|S'_3|=d_{G_{m_1}}(u_1)\geq \delta(G_{m_1})\geq \frac{n}{2}$ and $|S'_4|=d_{G_{m_3}}(u_p)\geq \delta(G_{m_3})\geq \frac{n}{2}$.

Now we show that $S'_3\cap S'_4=\emptyset$. Otherwise, $C'=u_1Pu_su_pPu_{s+1}u_1$ is a partial $\mathbf{G}$-transversal isomorphic to a cycle, where $s\in S'_3\cap S'_4$ and $C'$ has an associated injective $\varphi_3$ that arises from $\varphi$ by setting $\varphi_3(u_su_p)=m_3$ and $\varphi_3(u_1u_{s+1})=m_1$.

For any $t\in[2n-1]\setminus im(\varphi_3)$, by the maximality of $|V(P)|$, we have $N_{G_t}(y')\subseteq X'$ and $N_{G_t}(x')\subseteq Y'$, thus $|X'|\geq \frac{n}{2}$ and $|Y'|\geq \frac{n}{2}$. Since $N_{G_{m_1}}(u_1)\subseteq V(P)\cap Y$ and $N_{G_{m_3}}(u_p)\subseteq V(P)\cap X$, $d_{G_{m_1}}(u_1)\geq \frac{n}{2}$ and $d_{G_{m_3}}(u_p)\geq \frac{n}{2}$, we obtain $|V(P)\cap X|\geq \frac{n}{2}$ and $|V(P)\cap Y|\geq \frac{n}{2}$, which implies $|V(P)\cap X|=|V(P)\cap Y|=\frac{n}{2}$ and $p=n$ by $n=|X|=|X'|+|V(P)\cap X|=|Y|=|Y'|+|V(P)\cap Y|$. Furthermore, we have $G_t=K_{\frac{n}{2},\frac{n}{2}}\cup K_{\frac{n}{2},\frac{n}{2}}$ by $E_{G_t}[V(P)\cap X,Y']=E_{G_t}[V(P)\cap Y,X']=\emptyset$ and $\delta(G_t)\geq \frac{n}{2}$. For any edge $e\in C'$ with $\varphi_3(e)=k$, we obtain $\varphi_3'$ from $\varphi_3$ by replacing $\varphi_3(e)=k$ with $\varphi_3'(e)=s$. Then we have $G_k=K_{\frac{n}{2},\frac{n}{2}}\cup K_{\frac{n}{2},\frac{n}{2}}$. The arbitrariness of $k$ implies that $G_t=K_{\frac{n}{2},\frac{n}{2}}\cup K_{\frac{n}{2},\frac{n}{2}}$ for any $t\in im(\varphi_3)$. Thus $G_1=\cdots=G_{2n-1}=K_{\frac{n}{2},\frac{n}{2}}\cup K_{\frac{n}{2},\frac{n}{2}}$, which implies $(\mathrm{ii})$ holds, a contradiction.


Combining above arguments, $S'_3\cap S'_4=\emptyset$.
Since $|S'_3|\geq \frac{n}{2}$ and $|S'_4|\geq \frac{n}{2}$, we have $n\leq|S'_3|+|S'_4|=|S'_3\cup S'_4|\leq \frac{p}{2}<\frac{2n-2}{2}=n-1$, a contradiction.

This completes the proof of Claim 3.1.
\end{proof}

\noindent\textbf{Claim 3.2.}
Let $X=\{v_1,v_3,\dots,v_{2n-3},v_{2n-1}\}$, $Y=\{v_2,v_4,\dots,v_{2n-2},v_{2n}\}$, $P^*=v_1v_2\cdots \\v_{2n-2}$ be a partial $\mathbf{G}$-transversal with an associated injective $\psi$: $E(P^*)\to [2n-1]$ and $\{l_1,l_2\}=[2n-1]\setminus im(\psi)$. If $\mathbf{G}$ contains no partial transversal isomorphic to a cycle of order $2n-2$, then we have:
\begin{enumerate}[label={}, font=\normalfont, align=left, leftmargin=1em, labelsep=0em, itemindent=0em, itemsep=0.3ex]
    \item $(\mathrm{c_1})$ if $v_1v_{2n}\notin E(G_{l_1})$, then $v_{2n-2}v_{2n-1}\in E(G_{l_2})$.
    \item $(\mathrm{c_2})$ if $v_{2n-2}v_{2n-1}\notin E(G_{l_1})$, then $v_1v_{2n}\in E(G_{l_2})$.
\end{enumerate}

\begin{proof}
We only show $(\mathrm{c_1})$. The proof of $(\mathrm{c_2})$ is omitted since it can follow by the similar argument as $(\mathrm{c_1})$.

Let $v_1v_{2n}\notin E(G_{l_1})$ and $v_{2n-2}v_{2n-1}\notin E(G_{l_2})$.
Then we consider the following sets:
\[
S_5=\{s\in\{1,3,\dots,2n-3\}:v_{2n-2}v_{s}\in E(G_{l_2})\},
\]
\[
S_6=\{s\in\{1,3,\dots,2n-3\}:v_1v_{s+1}\in E(G_{l_1})\}.
\]
Since $v_{2n-2}v_{2n-1}\notin E(G_{l_2})$ and $v_1v_{2n}\notin E(G_{l_1})$, we have $|S_5|=d_{G_{l_2}}(v_{2n-2})\geq\left \lceil \frac{n}{2}  \right \rceil$, $|S_6|=d_{G_{l_1}}(v_1)\geq \left \lceil \frac{n}{2}  \right \rceil$, and thus $|S_5|+|S_6|\geq n$. Then $S_5\cap
S_6\neq\emptyset$ by $|S_5\cup S_6|\leq n-1$. Thus there exists $s\in\{1,3,\dots,2n-3\}$ such that $v_{2n-2}v_s\in E(G_{l_2})$ and $v_1v_{s+1}\in E(G_{l_1})$. It follows that $C=v_1P^*v_sv_{2n-2}P^*v_{s+1}v_1$ is a partial $\mathbf{G}$-transversal isomorphic to a cycle of order $2n-2$, where $C$ has an associated injective $\psi_1$ that arises $\psi$ by setting $\psi_1(v_1v_{s+1})=l_1$ and $\psi_1(v_{2n-2}v_s)=l_2$, a contradiction.

Combining above arguments, Claim 3.2 holds.
\end{proof}

\noindent\textbf{Claim 3.3.}
$\mathbf{G}$ contains a partial transversal isomorphic to a cycle of order $2n-2$.

\begin{proof}
Assume for a contradiction that there does not exist a partial $\mathbf{G}$-transversal isomorphic to a cycle of order $2n-2$.

By Claim 3.1, let $P=u_1u_2\cdots u_{2n-2}$ be a partial $\mathbf{G}$-transversal with an associated injective $\sigma$: $E(P)\to [2n-1]$ such that $\sigma(u_iu_{i+1})=i$ for $i\in[2n-3]$ and $\{2n-2,2n-1\}=[2n-1]\setminus im(\sigma)$. Without loss of generality, let $X=\{u_1,u_3,\dots,u_{2n-3},x\}$ and $Y=\{u_2,u_4,\dots,u_{2n-2},y\}$. Clearly, $u_1u_{2n-2}\notin E(G_{2n-2})\cup E(G_{2n-1})$.

If $u_{2n-2}x\in E(G_{2n-2})$ and $u_1y\in E(G_{2n-1})$, then $yu_1Pu_{2n-2}x$ is a $\mathbf{G}$-transversal isomorphic to a Hamiltonian path, which implies $(\mathrm{i})$ holds, a contradiction.
Thus either $u_{2n-2}x\notin E(G_{2n-2})$ or $u_1y\notin E(G_{2n-1})$.

Without loss of generality, we assume $u_{2n-2}x\notin E(G_{2n-2})$. Then $u_1y\in E(G_{2n-1})$ by taking $\{l_1,l_2\}=\{2n-2,2n-1\}$, $\{v_{2n-1},v_{2n}\}=\{x,y\}$ and identifying $P$ as $P^*$ in Claim 3.2.
Now we show $u_{2n-4}x\notin E(G_{2n-2})$. If not, $P'=yu_1Pu_{2n-4}x$ is a partial $\mathbf{G}$-transversal with an associated injective $\sigma_1$ that arises from $\sigma$ by setting $\sigma_1(u_1y)=2n-1$ and $\sigma_1(u_{2n-4}x)=2n-2$. In addition, $\{2n-4,2n-3\}=[2n-1]\setminus im(\sigma_1)$, $V\setminus V(P')=\{u_{2n-3},u_{2n-2}\}$.

If $u_{2n-2}x\in E(G_{2n-4})$, then $u_{2n-3}u_{2n-2}xP'y$ is a  $\mathbf{G}$-transversal isomorphic to a Hamiltonian path since $\sigma(u_{2n-3}u_{2n-2})=2n-3$, which implies $(\mathrm{i})$ holds, a contradiction.
If $u_{2n-2}x\notin E(G_{2n-4})$, we take $\{l_1,l_2\}=\{2n-4,2n-3\}$, $\{v_{2n-1},v_{2n}\}=\{u_{2n-3},u_{2n-2}\}$ and identify $P'$ as $P^*$ in Claim 2. Then $u_{2n-3}y\in E(G_{2n-3})$ by $(\mathrm{c_1})$ of Claim 3.2.
It follows that $C=yu_1Pu_{2n-3}y$ is a partial $\mathbf{G}$-transversal isomorphic to a cycle of order $2n-2$, where $C$ has an associated injective $\sigma_2$ that arises from $\sigma$ by setting $\sigma_2(u_1y)=2n-1$ and $\sigma_2(u_{2n-3}y)=2n-3$, a contradiction. Therefore, $u_{2n-4}x\notin E(G_{2n-2})$.

Clearly, $xy\notin E(G_{2n-2})$. Otherwise, $xyu_1Pu_{2n-2}$ is a  $\mathbf{G}$-transversal isomorphic to a Hamiltonian path, which implies $(\mathrm{i})$ holds, a contradiction. Since $u_1u_{2n-2}\notin E(G_{2n-1})$, $u_1y\in E(G_{2n-1})$ and $u_{2n-4}x,u_{2n-2}x,xy\notin E(G_{2n-2})$, we define the sets $S_7$ and $S_8$ as follows:
\[
S_7=\{s\in\{2,4,\dots,2n-6\}:u_1u_{s+2}\in E(G_{2n-1})\},
\]
\[
S_8=\{s\in\{2,4,\dots,2n-6\}:xu_s\in E(G_{2n-2})\},
\]
then $|S_7|\geq d_{G_{2n-1}}(u_1)-2\geq \left \lceil \frac{n}{2}  \right \rceil-2$ and $|S_8|=d_{G_{2n-2}}(x)\geq \left \lceil \frac{n}{2}  \right \rceil$. Thus $|S_7|+|S_8|\geq n-2$, which implies $S_7\cap S_8\neq\emptyset$ by $|S_7\cup S_8|\leq n-3$.

Let $s\in S_7\cap S_8\subseteq \{2,4,\dots,2n-6\}$. Then $u_1u_{s+2}\in E(G_{2n-1})$ and $xu_s\in E(G_{2n-2})$.

Firstly, $u_{2n-2}u_{s+1}\notin E(G_{s+1})$. If not, $u_1Pu_{s+1}u_{2n-2}Pu_{s+2}u_1$ is a partial $\mathbf{G}$-transversal isomorphic to a cycle of order $2n-2$, a contradiction.

Secondly, $P''=xu_sPu_1u_{s+2}Pu_{2n-2}$ is a partial $\mathbf{G}$-transversal with an associated injective $\sigma_3$ that arises from $\sigma$ by setting $\sigma_3(xu_s)=2n-2$ and $\sigma_3(u_1u_{s+2})=2n-1$, $\{s,s+1\}=[2n-1]\setminus im(\sigma_3)$ and $V\setminus V(P'')=\{u_{s+1},y\}$.
Let $\{l_1,l_2\}=\{s,s+1\}$, $\{v_{2n-1},v_{2n}\}=\{u_{s+1},y\}$ and identify $P''$ as $P^*$ in Claim 3.2. Then $xy\in E(G_s)$ by $u_{2n-2}u_{s+1}\notin E(G_{s+1})$ and $(\mathrm{c_2})$ of Claim 3.2.

Thirdly, $u_{s+1}y\notin E(G_{s+1})$. Otherwise, $u_{s+1}yxP''u_{2n-2}$ is a  $\mathbf{G}$-transversal isomorphic to a Hamiltonian path, which implies $(\mathrm{i})$ holds, a contradiction.

Finally, $P'''=u_{s+1}Pu_1u_{s+2}Pu_{2n-2}$ is a partial $\mathbf{G}$-transversal with an associated injective $\sigma_4$ that arises from $\sigma$ by setting $\sigma_4(u_1u_{s+2})=2n-1$, $\{s+1,2n-2\}=[2n-1]\setminus im(\sigma_4)$ and $V\setminus V(P''')=\{x,y\}$. Let $\{l_1,l_2\}=\{s+1,2n-2\}$, $\{v_{2n-1},v_{2n}\}=\{x,y\}$ and identify $P'''$ as $P^*$ in Claim 3.2. Since $u_{s+1}y\notin E(G_{s+1})$, we have $u_{2n-2}x\in E(G_{2n-2})$ by $(\mathrm{c_1})$ of Claim 3.2, this contradicts the assumption that $u_{2n-2}x\notin E(G_{2n-2})$.

Therefore, we complete the proof of Claim 3.3.
\end{proof}

By Claim 3.3, $\mathbf{G}$ contains a partial $\mathbf{G}$-transversal isomorphic to a cycle of order $2n-2$. Let $X$, $Y$, $C$, $D=(V,A)$ and $\phi$ be defined as in Definition \ref{def1-1}.

Since $d_{G_s}(u_s)\geq \left \lceil \frac{n}{2}  \right \rceil$ for any $u_s\in X\setminus\{x\}$, we have $d^{+}_{D}(u_s)\geq \left \lceil \frac{n}{2}  \right \rceil -1$, and thus
\begin{equation}\tag{$2$}\label{eq:1}
|A|=\sum_{u\in Y}d^{-}_{D}(u)=\sum_{u\in X\setminus\{x\}}d^{+}_{D}(u)\geq(n-1)(\left \lceil \frac{n}{2}  \right \rceil-1).
\end{equation}

\noindent\textbf{Claim 3.4.}
$xy\notin E(G_{2n-1})$.

\begin{proof}
By contradiction, we assume $xy\in E(G_{2n-1})$.

Firstly, we have $d^{-}_{D}(y)=0$. Otherwise, there exists a vertex $u_s\in X\setminus\{x\}$ such that $u_sy\in E(G_s)$, and thus $xyu_sC^{-}u_{s+1}$ is a $\mathbf{G}$-transversal isomorphic to a Hamiltonian path, which implies $(\mathrm{i})$ holds, a contradiction.

For any $t\in\{2,4,\dots,2n-2\}$, if $u_s\in N_{G_t}(x)\cap V(C)$, then $s\neq t$ and $u_{s-1}u_t\notin E(G_{s-1})$, say, $u_{s-1}\notin N^{-}_{D}(u_t)$. Otherwise, if $s=t$, then $xu_t\in E(G_t)$, and thus $yxu_tC^{-}u_{t+1}$ is a $\mathbf{G}$-transversal isomorphic to a Hamiltonian path, which implies $(\mathrm{i})$ holds, a contradiction; if $u_{s-1}u_t\in E(G_{s-1})$, then $yxu_sCu_tu_{s-1}C^{-}u_{t+1}$ is a $\mathbf{G}$-transversal isomorphic to a Hamiltonian path (see Figure 2 (c)) with an associated injective $\phi_1$ that arises from $\phi$ by setting $\phi_1(xy)=2n-1$, $\phi_1(xu_s)=t$ and $\phi_1(u_{s-1}u_t)=s-1$, which implies $(\mathrm{i})$ holds, a contradiction.
Therefore, $d^{-}_{D}(u_t)\leq |X\setminus\{x,u_{t-1}\}|-|N_{G_t}(x)\cap V(C)|\leq n-2-(\left \lceil \frac{n}{2}  \right \rceil -1)=\left \lfloor \frac{n}{2} \right \rfloor -1$. By $d^{-}_{D}(y)=0$, we have
\begin{equation}\tag{$3$}\label{eq:2}
\sum_{u\in Y}d^{-}_{D}(u)=d^{-}_{D}(y)+\sum_{t\in\{2,4,\dots,2n-2\}}d^{-}_{D}(u_t)\leq (n-1)(\left \lfloor \frac{n}{2} \right \rfloor -1).
\end{equation}

By \eqref{eq:1} and \eqref{eq:2}, we have $n$ is even, $d^{-}_{D}(u_t)=\frac{n}{2}-1$ and $|N_{G_t}(x)\cap V(C)|=\frac{n}{2}-1$, which implies $xy\in E(G_t)$ for any $t\in\{2,4,\dots,2n-2\}$ by $d_{G_t}(x)\geq\frac{n}{2}$.

Now we show $N_{G_{2n-1}}(x)\cap V(C)=\emptyset$. Otherwise, there exists some $t'\in\{2,4,\dots,2n-2\}$ such that $u_{t'}\in N_{2n-1}(x)\cap V(C)$, then $yxu_{t'}C^{-}u_{t'+1}$ is a $\mathbf{G}$-transversal isomorphic to a Hamiltonian path with an associated injective $\phi_2$ that arises from $\phi$ by setting $\phi_2(xy)=t'$ and $\phi_2(xu_{t'})=2n-1$, which implies $(\mathrm{i})$ holds, a contradiction.

Therefore, $d_{G_{2n-1}}(x)=1$ by $xy\in E(G_{2n-1})$ and $N_{G_{2n-1}}(x)\cap V(C)=\emptyset$, which is a contradiction to $\delta(G_{2n-1})\geq \frac{n}{2}>1$. This completes the proof of Claim 3.4.
\end{proof}

By Claim $3.4$, we have $|N_{G_{2n-1}}(x)\cap V(C)|=d_{G_{2n-1}}(x)\geq \delta(G_{2n-1})\geq \left \lceil \frac{n}{2}  \right \rceil$ and $|N_{G_{2n-1}}(x)\cap V(C)|=d_{G_{2n-1}}(y)\geq\delta(G_{2n-1})\geq \left \lceil \frac{n}{2}  \right \rceil$.

\noindent\textbf{Claim 3.5.}
$d^{-}_{D}(u)\leq \left \lfloor \frac{n}{2} \right \rfloor -1$ for any $u\in Y$.

\begin{proof}
We first show $d^{-}_{D}(y)\leq \left \lfloor \frac{n}{2} \right \rfloor -1$. If not, $d^{-}_{D}(y)\geq \left \lfloor \frac{n}{2} \right \rfloor$, then there exists $s\in[2n-2]$ such that $u_s\in N^{-}_{D}(y)$ and $u_{s+1}x\in E(G_{2n-1})$ since
$d_D^-(y)+|N_{G_{2n-1}}(x)\cap V(C)|\geq \left \lfloor \frac{n}{2} \right \rfloor+\left \lceil \frac{n}{2}  \right \rceil=n$. Therefore, $xu_{s+1}Cu_{s}y$ is a $\mathbf{G}$-transversal isomorphic to a Hamiltonian path with an associated injective $\phi_3$ that arises from $\phi$ by setting $\phi_3(u_{s+1}x)=2n-1$ and $\phi_3(u_{s}y)=s$, a contradiction.

Next, we show $d^{-}_{D}(u)\leq \left \lfloor \frac{n}{2} \right \rfloor-1$ for any $u\in Y\setminus\{y\}$. Suppose that there exists a vertex $u\in Y\setminus\{y\}$ such that $d^{-}_{D}(u)\geq \left \lfloor \frac{n}{2} \right \rfloor $. Without loss of generality, let $d^{-}_{D}(u_2)\geq \left \lfloor \frac{n}{2} \right \rfloor $.
Since $N_{G_{2n-1}}(x)\subseteq V(C)\cap Y$ and $N_{G_{2}}(y)\subseteq X$, we have $|(N_{G_{2n-1}}(x)\cup N_{G_2}(y))\cap V(C)|\geq d_{G_{2n-1}}(x)+d_{G_2}(y)-1\geq\left \lceil \frac{n}{2}  \right \rceil+(\left \lceil \frac{n}{2}  \right \rceil-1)\geq n-1$. By Lemma \ref{lem2-1}, there exists a $\mathbf{G}$-transversal isomorphic to a Hamiltonian path between $x$ and $y$, which implies $(\mathrm{i})$ holds, a contradiction.

Combining above arguments, we have $d^{-}_{D}(u)\leq \left \lfloor \frac{n}{2} \right \rfloor -1$ for any $u\in Y$.
\end{proof}

\noindent\textbf{Claim 3.6.} There are at least $\left \lfloor \frac{n}{2} \right \rfloor +1$ vertices in $Y$ for which the indegree is $\left \lfloor \frac{n}{2} \right \rfloor -1$.

\begin{proof}
If not, there are at most $\left \lfloor \frac{n}{2} \right \rfloor$ vertices in $Y$ for which the indegree is $\left \lfloor \frac{n}{2} \right \rfloor -1$, thus
\[\sum_{u\in Y}d^{-}_{D}(u)\leq \left \lfloor \frac{n}{2} \right \rfloor (\left \lfloor \frac{n}{2} \right \rfloor -1)+(n-\left \lfloor \frac{n}{2} \right \rfloor )(\left \lfloor \frac{n}{2} \right \rfloor -2)<(n-1)(\left \lceil \frac{n}{2}  \right \rceil-1).
\]
This contradicts \eqref{eq:1}.
\end{proof}

By Claim $3.6$, there are at least $\left \lfloor \frac{n}{2} \right \rfloor$ vertices in $Y\setminus\{y\}$ for which the indegree is equal to $\left \lfloor \frac{n}{2} \right \rfloor -1$, and thus there exists $s\in\{2,4,\dots,2n-2\}$ such that $d^{-}_{D}(u_s)=\left \lfloor \frac{n}{2} \right \rfloor-1$ and $u_{s+1}y\in E(G_{2n-1})$ by $|N_{G_{2n-1}}(y)\cap V(C)|\geq \left \lceil \frac{n}{2}  \right \rceil$ and $\left \lfloor \frac{n}{2} \right \rfloor+\left \lceil \frac{n}{2}  \right \rceil=n$.

If $xy\in E(G_s)$, then $xyu_{s+1}Cu_s$ is a $\mathbf{G}$-transversal isomorphic to a Hamiltonian path, which implies $(\mathrm{i})$ holds, a contradiction.
Thus $|N_{G_s}(x)\cap V(C)|=d_{G_s}(x)$ and $d^{-}_{D}(u_s)+|N_{G_s}(x)\cap V(C)|\geq\left \lfloor \frac{n}{2} \right \rfloor-1+ \left \lceil \frac{n}{2}  \right \rceil=n-1$. By Lemma \ref{lem2-2}, there exists a $\mathbf{G}$-transversal isomorphic to a Hamiltonian path between $x$ and $y$, which implies $(\mathrm{i})$ holds, a contradiction.

Therefore, either $(\mathrm{i})$ or $(\mathrm{ii})$ holds, which completes the proof of Theorem \ref{thm1-3}.
{\hfill $\square$ \par}
\section{The proof of Theorem \ref{thm1-4}}
\hspace{1.5em}In this section, we prove Theorem \ref{thm1-4}. For $n=1,2$, Theorem \ref{thm1-4} holds trivially, thus we assume $n\geq3$ in the following proof.

\noindent\textbf{\textit{Proof of Theorem \ref{thm1-4}}.}
Choose any two vertices $x\in X$ and $y\in Y$. We will show that there is a $\mathbf{G}$-transversal isomorphic to a Hamiltonian path between $x$ and $y$ or $\mathbf{G}=\mathbf{F_{2n-1}}$.

Let $H_i=G_i-\{x,y\}$ for each $i\in[2n-1]$ and $\mathbf{H}=\{H_1,\dots,H_{2n-1}\}$ be a bipartite graph collection with the same bipartition $V_{H}=(X\setminus\{x\},Y\setminus\{y\})$ with $|X\setminus\{x\}|=|Y\setminus\{y\}|=n-1$. Then $\delta(H_i)\geq\left \lceil \frac{n-1}{2}  \right \rceil$ for each $i\in[2n-1]$. By Theorem \ref{thm1-3}, one of the following two cases holds: $(\mathrm{1})$ There exists a partial $\mathbf{H}$-transversal isomorphic to a path of order $2n-2$; $(\mathrm{2})$ $n$ is odd and $H_1=\dots=H_{2n-1}=K_{\frac{n-1}{2},\frac{n-1}{2}}\cup K_{\frac{n-1}{2},\frac{n-1}{2}}$. Now we consider these two cases by showing Claims A and B.

\noindent\textbf{Claim A.} If $n$ is odd and $H_1=\dots=H_{2n-1}=K_{\frac{n-1}{2},\frac{n-1}{2}}\cup K_{\frac{n-1}{2},\frac{n-1}{2}}$, then $\mathbf{G}=\mathbf{F_{2n-1}}$.

\begin{proof}
Without loss of generality, let $X\setminus\{x\}=X_1\cup X_2$ and $Y\setminus\{y\}=Y_1\cup Y_2$, where $X_1=\{x_1,x_2,\dots,x_{\frac{n-1}{2}}\}$, $X_2=\{x_{\frac{n+1}{2}},x_{\frac{n+3}{2}},\dots,x_{n-1}\}$,
$Y_1=\{y_1,y_2,\dots,y_{\frac{n-1}{2}}\}$ and $Y_2=\{y_{\frac{n+1}{2}},y_{\frac{n+3}{2}},\dots,y_{n-1}\}$.
In addition, we take $H_i[X_1,Y_1]=K_{\frac{n-1}{2},\frac{n-1}{2}}$ and $H_i[X_2,Y_2]=K_{\frac{n-1}{2},\frac{n-1}{2}}$ for each $i\in[2n-1]$.
Since $\delta(G_i)\geq\frac{n+1}{2}$ for each $i\in[2n-1]$, we have $x_jy,xy_j\in E(G_i)$ for each $j\in[n-1]$, which implies $G_i=F$ or $G_i=F\cup\{xy\}$, where $F$ is shown in Figure 1.
Therefore, $\mathbf{G}=\mathbf{F_{2n-1}}$.
\end{proof}

\noindent\textbf{Claim B.} If there exists a partial $\mathbf{H}$-transversal isomorphic to a path of order $2n-2$, then there is a $\mathbf{G}$-transversal isomorphic to a Hamiltonian path between $x$ and $y$.

\begin{proof}
We define: $(\mathrm{d_1})$ $\mathbf{G}$ contains a transversal isomorphic to a Hamiltonian path between $x$ and $y$; $(\mathrm{d_2})$ $\mathbf{H}$ contains a partial transversal isomorphic to a cycle of order $2n-2$.

To prove that $(\mathrm{d_1})$ holds, we first show that either $(\mathrm{d_1})$ or $(\mathrm{d_2})$ holds, and then that $(\mathrm{d_2})$ implies $(\mathrm{d_1})$.

\noindent\textbf{Claim B.1.} Either $(\mathrm{d_1})$ or $(\mathrm{d_2})$ holds.

\begin{proof}
Clearly, one partial $\mathbf{H}$-transversal implies there exists a corresponding partial $\mathbf{G}$-transversal. Since there exists a partial $\mathbf{H}$-transversal isomorphic to a path of order $2n-2$, we can take $P=u_1u_2\cdots u_{2n-2}$ to be a partial $\mathbf{G}$-transversal with an associated injective $\tau$: $E(P)\to [2n-1]$ such that $\tau(u_iu_{i+1})=i$ for $i\in[2n-3]$ and $\{2n-2,2n-1\}=[2n-1]\setminus im(\tau)$.
Without loss of generality, let $X=\{u_1,u_3,\dots,u_{2n-3},x\}$ and $Y=\{u_2,u_4,\dots,u_{2n-2},y\}$.

$\mathbf{Case~1:}$ $u_1u_{2n-2}\in E(G_{2n-2})\cup E(G_{2n-1})$.

Then $\mathbf{H}$ contains a partial transversal isomorphic to a cycle of order $2n-2$, thus $(\mathrm{d_2})$ holds.

$\mathbf{Case~2:}$ $u_1u_{2n-2}\notin E(G_{2n-2})\cup E(G_{2n-1})$.

$\mathbf{Case~2.1:}$ $u_{2n-2}x\in E(G_{2n-1})$ and $u_1y\in E(G_{2n-2})$.

Then $yu_1Pu_{2n-2}x$ is a $\mathbf{G}$-transversal isomorphic to a Hamiltonian path between $x$ and $y$, thus $(\mathrm{d_1})$ holds.

$\mathbf{Case~2.2:}$ $u_{2n-2}x\notin E(G_{2n-1})$ or $u_1y\notin E(G_{2n-2})$.

Without loss of generality, we assume $u_{2n-2}x\notin E(G_{2n-1})$. Now we define the sets $S_9$ and $S_{10}$ as follows:
\[
S_9=\{s\in\{3,5,\dots,2n-3\}:u_{2n-2}u_s\in E(G_{2n-1})\},
\]
\[
S_{10}=\{s\in\{1,3,\dots,2n-5\}:u_1u_{s+1}\in E(G_{2n-2})\}.
\]
Since $u_1u_{2n-2}\notin E(G_{2n-2})\cup E(G_{2n-1})$ and $u_{2n-2}x\notin E(G_{2n-1})$, we have $|S_9|=d_{G_{2n-1}}(u_{2n-2})\\\geq\left \lceil \frac{n+1}{2}  \right \rceil$ and $|S_{10}|\geq d_{G_{2n-2}}(u_1)-1\geq\left \lceil \frac{n+1}{2}  \right \rceil-1=\left \lceil \frac{n-1}{2}  \right \rceil$, then $S_9\cap S_{10}\neq\emptyset$ by $|S_9\cup S_{10}|\leq n-1<n\leq|S_9|+|S_{10}|$. Thus there exists $s\in\{3,\dots,2n-5\}$ such that $u_{2n-2}u_s\in E(G_{2n-1})$ and $u_1u_{s+1}\in E(G_{2n-2})$. Therefore, $u_1Pu_su_{2n-2}Pu_{s+1}u_1$ is a partial $\mathbf{H}$-transversal isomorphic to a cycle of order $2n-2$ with an associated injective $\tau_1$ that arises from $\tau$ by setting $\tau_1(u_{2n-2}u_s)=2n-1$ and $\tau_1(u_1u_{s+1})=2n-2$, thus $(\mathrm{d_2})$ holds.

Therefore, Claim B.1 holds.
\end{proof}

\noindent\textbf{Claim B.2.} If $(\mathrm{d_2})$ holds, then $(\mathrm{d_1})$ holds.

\begin{proof}
Since $(\mathrm{d_2})$ holds, there exists a $\mathbf{G}$-transversal isomorphic to a cycle of order $2n-2$. Let $X$, $Y$, $C$, $D=(V,A)$ and $\phi$ be defined as in Definition \ref{def1-1}.
Then $d^{+}_{D}(u_s)\geq \left \lceil \frac{n-1}{2}  \right \rceil=\left \lfloor \frac{n}{2}  \right \rfloor$ by $\delta(G_s)\geq\left \lceil \frac{n+1}{2}  \right \rceil$ for any $u_s\in X\setminus\{x\}$.

$\mathbf{Case~1:}$ $d^{-}_{D}(y)\geq\left \lfloor \frac{n}{2} \right \rfloor$.

Let $u_t\in N_{G_{2n-1}}(x)\cap V(C)$.
If there exists $u_s\in N^{-}_{D}(y)$ such that $u_{s+1}u_{t+1}\in E(G_t)$, then $xu_tC^{-}u_{s+1}u_{t+1}Cu_sy$ is a $\mathbf{G}$-transversal isomorphic to a Hamiltonian path between $x$ and $y$ with an associated injective $\phi_1$ that arises from $\phi$ by setting $\phi_1(xu_t)=2n-1$, $\phi_1(u_{s+1}u_{t+1})=t$ and $\phi_1(u_sy)=s$, thus $(\mathrm{d_1})$ holds.

Now we consider that for any $u_s\in N^{-}_{D}(y)$, $u_{s+1}u_{t+1}\notin E(G_t)$. We have
\[|N_{G_t}(u_{t+1})\cap V(C)|\leq n-1-d^{-}_{D}(y)\leq \left \lceil \frac{n}{2}  \right \rceil-1.
\]
Moreover, $|N_{G_t}(u_{t+1})\cap V(C)|\geq  \left \lceil \frac{n-1}{2}  \right \rceil$ by $d_{G_t}(u_{t+1})\geq \left \lceil \frac{n+1}{2}  \right \rceil$, which implies $n$ is odd and $|N_{G_t}(u_{t+1})\cap V(C)|=\frac{n-1}{2}$, and thus $u_{t+1}y\in E(G_t)$. It follows that $xu_tC^{-}u_{t+1}y$ is a $\mathbf{G}$-transversal isomorphic to a Hamiltonian path between $x$ and $y$ with an associated injective $\phi_2$ that arises from $\phi$ by setting $\phi_2(xu_t)=2n-1$ and $\phi_2(u_{t+1}y)=t$, thus $(\mathrm{d_1})$ holds.

$\mathbf{Case~2:}$ $d^{-}_{D}(y)\leq\left \lfloor \frac{n}{2} \right \rfloor -1$.

If there exists $u\in Y\setminus\{y\}$ such that $d^{-}_{D}(u)\geq\left \lceil \frac{n}{2} \right \rceil$, without loss of generality, we assume $d^{-}_{D}(u_2)\geq\left \lceil \frac{n}{2} \right \rceil$. Since $|(N_{G_{2n-1}}(x)\cup N_{G_2}(y))\cap V(C)|=|N_{G_{2n-1}}(x)\cap V(C)|+|N_{G_2}(y)\cap V(C)|\geq \left \lceil \frac{n-1}{2} \right \rceil+\left \lceil \frac{n-1}{2} \right \rceil\geq n-1$, there exists a $\mathbf{G}$-transversal isomorphic to a Hamiltonian
path between $x$ and $y$ by Lemma \ref{lem2-1}, thus $(\mathrm{d_1})$ holds.

If $d^{-}_{D}(u)\leq\left \lceil \frac{n}{2} \right \rceil-1$ for any $u\in Y\setminus\{y\}$, then there are at least $\left \lceil \frac{n}{2} \right \rceil$ vertices in $Y\setminus\{y\}$ for which the indegree is $\left \lceil \frac{n}{2} \right \rceil-1$. Otherwise, we have
\begin{align*}
\sum_{u\in Y}d^{-}_{D}(u)
&=d^{-}_{D}(y)+\sum_{u\in Y\setminus\{y\}}d^{-}_{D}(u)\\
&\leq(\left \lfloor \frac{n}{2}  \right \rfloor-1)+(\left \lceil \frac{n}{2} \right \rceil-1)(\left \lceil \frac{n}{2} \right \rceil-1)+(n-\left \lceil \frac{n}{2} \right \rceil)(\left \lceil \frac{n}{2} \right \rceil-2)\\
&<(n-1)\left \lfloor \frac{n}{2}\right \rfloor,
\end{align*}
which implies a contradiction by $\displaystyle\sum_{u\in Y}d^{-}_{D}(u)=\sum_{u\in X\setminus\{x\}}d^{+}_{D}(u)\geq(n-1) \left \lfloor \frac{n}{2}  \right \rfloor$.

Since there are at least $\left \lceil \frac{n}{2} \right \rceil$ vertices in $Y\setminus\{y\}$ for which the indegree is $\left \lceil \frac{n}{2} \right \rceil-1$, $|N_{G_{2n-1}}(y)\cap V(C)|\geq \left \lceil \frac{n+1}{2} \right \rceil-1=\left \lfloor \frac{n}{2}  \right \rfloor$ and $\left \lceil \frac{n}{2} \right \rceil+\left \lfloor \frac{n}{2}  \right \rfloor=n$, there exists $k\in\{2,4,\dots,2n-2\}$ such that $d^{-}_{D}(u_k)=\left \lceil \frac{n}{2} \right \rceil-1$ and $u_{k+1}y\in E(G_{2n-1})$. Moreover, since $d^{-}_{D}(u_k)+|N_{G_k}(x)\cap V(C)|\geq \left \lceil \frac{n}{2} \right \rceil-1+\left \lfloor \frac{n}{2}  \right \rfloor=n-1$, there exists a $\mathbf{G}$-transversal isomorphic to a Hamiltonian path between $x$ and $y$ by Lemma \ref{lem2-2}, thus $(\mathrm{d_1})$ holds.

Combining the two cases, $(\mathrm{d_1})$ holds, then Claim B.2 holds.
\end{proof}
By Claims B.1 and B.2, Claim B holds.
\end{proof}

By Claims A and B, we complete the proof of Theorem \ref{thm1-4}.
{\hfill $\square$ \par}

\section{The proof of Theorem \ref{thm1-5}}
\hspace{1.5em}In this section, we prove Theorem \ref{thm1-5}. For $n=1,2$, Theorem \ref{thm1-5} holds trivially, thus we assume $n\geq3$ in the following proof.

\noindent\textbf{\textit{Proof of Theorem \ref{thm1-5}}.}
Let $|X|=n+1$ and $|Y|=n$. We start the proof by establishing the following claims.

\noindent\textbf{Claim 5.1.} $\mathbf{G}$ contains a partial transversal
isomorphic to a path of order $2n$.

\begin{proof}
Let $P$ be a (partial) $\mathbf{G}$-transversal isomorphic to a path such that $|V(P)|$ is maximum.
If $|V(P)|\geq 2n$, then Claim 5.1 holds.

Now we assume that $P=u_1u_2\cdots u_k$ with $|V(P)|=k<2n$. Then $|V\setminus V(P)|\geq 2$. Without loss of generality, let $P$ have an associated injective $\varphi$: $E(P)\to [2n]$ with $\varphi(u_iu_{i+1})=i$ for $i\in[k-1]$ and $\{f_1,f_2\}\subseteq[2n]\setminus im(\varphi)=[k,2n]$.

$\mathbf{Case~1:}$ $u_1,u_k$ belong to the same partition.

By $|V(P)|<2n$, if $u_1,u_k\in X$, then $Y\setminus (V(P)\cap Y)\neq \emptyset$; similarly, if $u_1,u_k\in Y$, then $X\setminus (V(P)\cap X)\neq \emptyset$. Accordingly, we define $z\in Y\setminus (V(P)\cap Y)$ if $u_1,u_k\in X$, and $z\in X\setminus (V(P)\cap X)$ if $u_1,u_k\in Y$. Let $\{Z_1,Z_2\}=\{X,Y\}$ such that $z\in Z_1$.

Consider the following sets:
\[
J_1=\{j\in\{1,3,\dots,k-2\}:u_1u_{j+1}\in E(G_{f_1})\},
\]
\[
J_2=\{j\in\{1,3,\dots,k-2\}:zu_j\in E(G_{f_2})\}.
\]

Clearly, $J_1\cap J_2=\emptyset$. Otherwise, we take $j\in J_1\cap J_2$, then $P'=zu_jPu_1u_{j+1}Pu_k$ is a partial $\mathbf{G}$-transversal isomorphic to a path that is longer than $P$, where $P'$ has an associated injective $\varphi_1$ that arises from $\varphi$ by setting $\varphi_1(zu_j)=f_2$ and $\varphi_1(u_1u_{j+1})=f_1$, a contradiction.

By the maximality of $|V(P)|$, we have $N_{G_{f_1}}(u_1)\subseteq V(P)$ and $u_1z,u_kz\notin E(G_{f_2})$. Thus $|J_1|=d_{G_{f_1}}(u_1)\geq\left \lceil \frac{n+1}{2}  \right \rceil$. Since $J_1\cap J_2=\emptyset$ and $u_kz\notin E(G_{f_2})$, we have $d_{G_{f_2}}(z)\leq |Z_2|-|J_1|-1\leq |X|-|J_1|-1\leq\left \lfloor \frac{n-1}{2} \right \rfloor$, which is a contradiction to $\delta(G_{f_2})\geq \left \lceil \frac{n+1}{2}  \right \rceil$.

$\mathbf{Case~2:}$ $u_1,u_k$ belong to different partitions.

Without loss of generality, we assume $u_1\in X$. Clearly, $X\setminus (V(P)\cap X)\neq \emptyset$ by $|V(P)|<2n$. Then we take $x'\in X\setminus (V(P)\cap X)$. Now we consider the following sets:
\[
J_3=\{j\in\{1,3,\dots,k-1\}:u_ku_j\in E(G_{f_1})\},
\]
\[
J_4=\{j\in\{1,3,\dots,k-1\}:x'u_{j+1}\in E(G_{f_2})\}.
\]

Similarly, $J_3\cap J_4=\emptyset$. Otherwise, $P''=x'u_{j+1}Pu_ku_jPu_1$ is a partial $\mathbf{G}$-transversal isomorphic to a path that is longer than $P$, where $j\in J_3\cap J_4$ and $P'$ has an associated injective $\varphi_2$ that arises from $\varphi$ by setting $\varphi_2(x'u_{j+1})=f_2$ and $\varphi_1(u_ku_j)=f_1$, a contradiction.

By the maximality of $|V(P)|$, we have $N_{G_{f_1}}(u_k)\subseteq V(P)$. Thus $|J_3|=d_{G_{f_1}}(u_k)\geq\left \lceil \frac{n+1}{2}  \right \rceil$. Since $J_3\cap J_4=\emptyset$, we have $d_{G_{f_2}}(x')\leq |Y|-|J_3|\leq\left \lfloor \frac{n-1}{2} \right \rfloor$, which is a contradiction to $\delta(G_{f_2})\geq \left \lceil \frac{n+1}{2}  \right \rceil$.

Therefore, Claim 5.1 holds.
\end{proof}

Without loss of generality, let $X=\{u_1,u_3,\dots,u_{2n-1},u_{2n+1}\}$, $Y=\{u_2,u_4,\dots,u_{2n}\}$ by $|X|=n+1$ and $|Y|=n$.

For convenience, we define $(\mathrm{e_1})$, $(\mathrm{e_2})$, $(\mathrm{e_3})$ as follows: $(\mathrm{e_1})$ $\mathbf{G}$ contains a partial transversal
isomorphic to the disjoint union of a cycle of order $2n-2$ and a $K_2$; $(\mathrm{e_2})$ $\mathbf{G}$ contains a partial transversal isomorphic to a cycle of order $2n$; $(\mathrm{e_3})$ $\mathbf{G}$ contains a transversal isomorphic to a Hamiltonian path.

To prove that $(\mathrm{e_3})$ holds, we first show that Claims 5.2-5.4 hold.

\noindent\textbf{Claim 5.2.} One of $(\mathrm{e_1})$, $(\mathrm{e_2})$, $(\mathrm{e_3})$ holds.

\begin{proof}
By Claim 5.1, let $P=u_1u_2\cdots u_{2n-1}u_{2n}$ be a partial $\mathbf{G}$-transversal with an associated injective $\psi$: $E(P)\to [2n]$ such that $\psi(u_iu_{i+1})=i$ for $i\in[2n-1]$ and $\{2n\}=[2n]\setminus im(\psi)$.

If $u_1u_{2n-2}\in E(G_{2n-2})\cup E(G_{2n})$, then $(P-\{u_{2n-2}u_{2n-1}\})\cup\{u_1u_{2n-2}\}$ is a partial transversal isomorphic to the disjoint union of a cycle of order $2n-2$ and a $K_2$, thus $(\mathrm{e_1})$ holds.

If $u_3u_{2n}\in E(G_2)\cup E(G_{2n})$, then $(P-\{u_2u_3\})\cup\{u_3u_{2n}\}$ is a partial transversal isomorphic to the disjoint union of a cycle of order $2n-2$ and a $K_2$, thus $(\mathrm{e_1})$ holds.

If $u_1u_{2n}\in E(G_{2n})$, then $\mathbf{G}$ contains a partial transversal isomorphic to a cycle of order $2n$, thus $(\mathrm{e_2})$ holds.

If $u_{2n}u_{2n+1}\in E(G_{2n})$, then $\mathbf{G}$ contains a transversal isomorphic to a Hamiltonian path, thus $(\mathrm{e_3})$ holds.

Next, we assume that $u_1u_{2n-2}\notin E(G_{2n-2})\cup E(G_{2n})$, $u_3u_{2n}\notin E(G_2)\cup E(G_{2n})$ and $u_1u_{2n},u_{2n}u_{2n+1}\notin E(G_{2n})$.
Then we consider the following sets:
\[
J_5=\{j\in\{3,5,\dots,2n-1\}:u_3u_{j+1}\in E(G_{2})\},
\]
\[
J_6=\{j\in\{3,5,\dots,2n-1\}:u_{2n}u_j\in E(G_{2n})\}.
\]

Since $u_2u_3\in E(G_2)$ and $u_1u_{2n},u_{2n}u_{2n+1}\notin E(G_{2n})$, we have $|J_5|=d_{G_{2}}(u_3)-1\geq\left \lceil \frac{n-1}{2}  \right \rceil$, $|J_6|=d_{G_{2n}}(u_{2n})\geq\left \lceil \frac{n+1}{2}  \right \rceil$, and thus $|J_5|+|J_6|\geq n$. Since $|J_5\cup J_6|\leq n-1$, we have $J_5\cap J_6\neq\emptyset$. It follows that there exists $j\in\{3,5,\dots,2n-1\}$ such that $u_3u_{j+1}\in E(G_{2})$ and $u_{2n}u_j\in E(G_{2n})$. Therefore, $(u_3Pu_{2n}-\{u_{j}u_{j+1}\})\cup\{u_3u_{j+1},u_{2n}u_{j},u_1u_2\}$ is a partial transversal isomorphic to the disjoint union of a cycle of order $2n-2$ and a $K_2$ with an associated injective $\psi_1$ that arises $\psi$ by setting $\psi_1(u_3u_{j+1})=2$ and $\psi_1(u_{2n}u_{j})=2n$, thus $(\mathrm{e_1})$ holds.

Combining above arguments, Claim 5.2 holds.
\end{proof}

\noindent\textbf{Claim 5.3.}
If $(\mathrm{e_1})$ holds, then $(\mathrm{e_2})$ or $(\mathrm{e_3})$ holds.

\begin{proof}
Without loss of generality, let $C=u_1u_2\dots u_{2n-2}u_1$ and $u_{2n-1}u_{2n}$ be a partial $\mathbf{G}$-transversal isomorphic to the disjoint union of a cycle of order $2n-2$ and a $K_2$ with an associated injective $\phi$: $E(C)\to [2n]$ such that $\phi(u_iu_{i+1})=i$ for $i\in[2n-3]\cup\{2n-1\}$, $\phi(u_{2n-2}u_1)=2n-2$ and $\{2n\}=[2n]\setminus im(\phi)$.

Let $X\setminus\{u_{2n-1},u_{2n+1}\}=X_1$, $D=(V,A)$ be a digraph on vertex set $V$ such that
\[A=\bigcup_{u_i\in X_1}\{\overrightarrow{u_iu}:u\in Y, u_iu\in E(G_i) \text{ and } u\neq u_{i+1}\}.\]
Then $d^{-}_{D}(u)\leq |X_1|=n-1$ for any $u\in Y$, and
\begin{equation}\tag{$4$}\label{eq:3}
|A|=\displaystyle\sum_{u\in Y}d^{-}_{D}(u)=\sum_{u\in X_1}d^{+}_{D}(u)\geq(n-1)\left \lceil \frac{n-1}{2}  \right \rceil=(n-1)\left \lfloor \frac{n}{2} \right \rfloor.
\end{equation}

$\mathbf{Case~1:}$ $u_{2n}u_{2n+1}\in E(G_{2n})$.

For any $t\in\{2,4,\dots,2n-2\}$, we have $|N_{G_t}(u_{2n+1})\cap V(C)|\geq d_{G_t}(u_{2n+1})-1\geq\left \lceil \frac{n+1}{2}  \right \rceil -1=\left \lfloor \frac{n}{2} \right \rfloor$.

If there exists $u_s\in N_{G_t}(u_{2n+1})\cap V(C)$ such that $u_{s-1}u_t\in E(G_{s-1})$, then $u_{2n-1}u_{2n}u_{2n+1}\\u_sCu_tu_{s-1}C^{-}u_{t+1}$ is a $\mathbf{G}$-transversal isomorphic to a Hamiltonian path with an associated injective $\phi_1$ that arises from $\phi$ by setting $\phi_1(u_{2n}u_{2n+1})=2n$, $\phi_1(u_{2n+1}u_s)=t$ and $\phi_1(u_{s-1}u_t)=s-1$, thus $(\mathrm{e_3})$ holds.

If $u_{s-1}u_t\notin E(G_{s-1})$ for any $u_s\in N_{G_t}(u_{2n+1})\cap V(C)$, then $s\neq t$ and $u_{s-1}\notin N^{-}_{D}(u_t)$. Thus $d^{-}_{D}(u_t)\leq |X\setminus\{u_{2n-1},u_{2n+1},u_{t-1}\}|-|N_{G_t}(u_{2n+1})\cap V(C)|\leq n-2-\left \lfloor \frac{n}{2} \right \rfloor=\left \lfloor \frac{n-3}{2} \right \rfloor$. It follows that
\begin{align*}
d^{-}_{D}(u_{2n})
=\sum_{u\in Y}d^{-}_{D}(u)-\sum_{t\in\{2,4,\dots,2n-2\}}d^{-}_{D}(u_t)
\geq (n-1)\left \lfloor \frac{n}{2} \right \rfloor-(n-1)\left \lfloor \frac{n-3}{2} \right \rfloor
\geq n-1.
\end{align*}
Thus $d^{-}_{D}(u_{2n})=n-1$ by $d^{-}_{D}(u_{2n})\leq n-1$, which implies $\{u_1,u_3,\dots,u_{2n-3}\}=N^{-}_{D}(u_{2n})$.

In addition, $|N_{G_{2n-1}}(u_{2n+1})\cap V(C)|\geq 1$ by $d_{G_{2n-1}}(u_{2n+1})\geq \left \lceil \frac{n+1}{2}  \right \rceil\geq2$. We assume $u_{s'}\in N_{G_{2n-1}}(u_{2n+1})\cap V(C)$. Then $u_{s'-1}\in N^{-}_{D}(u_{2n})$, and thus $u_{2n+1}u_{s'}Cu_{s'-1}u_{2n}u_{2n+1}$ is a partial transversal isomorphic to a cycle of order $2n$ with an associated injective $\phi_2$ that arises from $\phi$ by setting $\phi_2(u_{2n}u_{2n+1})=2n$, $\phi_2(u_{2n+1}u_{s'})=2n-1$ and $\phi_2(u_{s'-1}u_{2n})=s'-1$, thus $(\mathrm{e_2})$ holds.

$\mathbf{Case~2:}$ $u_{2n}u_{2n+1}\notin E(G_{2n})$.

If $d^{-}_{D}(u_{2n})=n-1$, similar to Case $1$, we take $u_s\in N_{G_{2n}}(u_{2n+1})\cap V(C)$, then $u_{s-1}\in N^{-}_{D}(u_{2n})$, and $u_{2n+1}u_sCu_{s-1}u_{2n}u_{2n-1}$ is a $\mathbf{G}$-transversal isomorphic to a Hamiltonian path with an associated injective $\phi_3$ that arises from $\phi$ by setting $\phi_3(u_{2n+1}u_s)=2n$ and $\phi_3(u_{s-1}u_{2n})=s-1$, thus $(\mathrm{e_3})$ holds.

If $d^{-}_{D}(u_{2n})<n-1$, then there exists $u\in Y\setminus\{u_{2n}\}$ such that $d^{-}_{D}(u)\geq\left \lfloor \frac{n}{2} \right \rfloor$. Otherwise, by \eqref{eq:3}, we have
\[(n-1)\left \lfloor \frac{n}{2} \right \rfloor\leq\sum_{u\in Y}d^{-}_{D}(u)=\sum_{u\in Y\setminus\{u_{2n}\}}d^{-}_{D}(u)+d^{-}_{D}(u_{2n})<(n-1)(\left \lfloor \frac{n}{2} \right \rfloor-1)+n-1=(n-1)\left \lfloor \frac{n}{2} \right \rfloor,
\]
a contradiction. Without loss of generality, we take $d^{-}_{D}(u_2)\geq \left \lfloor \frac{n}{2} \right \rfloor$.

In addition, $|N_{G_{2n}}(u_{2n+1})\cap V(C)|=d_{G_{2n}}(u_{2n+1})$ and $|N_{G_2}(u_{2n})\cap V(C)|\geq d_{G_2}(u_{2n})-2$. Thus $|(N_{G_{2n}}(u_{2n+1})\cup N_{G_2}(u_{2n}))\cap V(C)|=|N_{G_{2n}}(u_{2n+1})\cap V(C)|+|N_{G_2}(u_{2n})\cap V(C)|\geq \left \lceil \frac{n+1}{2} \right \rceil+\left \lceil \frac{n+1}{2} \right \rceil-2\geq n-1$.

When $|N_{G_2}(u_{2n})\cap V(C)|=0$, then $n=3$, $N_{G_2}(u_{2n})=\{u_{2n-1},u_{2n+1}\}$ and $|N_{G_{2n}}(u_{2n+1})\cap V(C)|\geq \left \lceil \frac{n+1}{2} \right \rceil=2$, thus $u_{2n+1}u_2\in E(G_{2n})$ by $n=3$. Therefore, $u_{2n-1}u_{2n}u_{2n+1}u_2C^{-}u_3$ is a $\mathbf{G}$-transversal isomorphic to a Hamiltonian path with an associated injective $\phi_4$ that arises from $\phi$ by setting $\phi_4(u_{2n}u_{2n+1})=2$ and $\phi_4(u_{2n+1}u_2)=2n$, thus $(\mathrm{e_3})$ holds.

When $|N_{G_2}(u_{2n})\cap V(C)|>0$, then there exists $u_t\in Y\cap V(C)$ such that $u_t\in N_{G_{2n}}(u_{2n+1})$ and $N_{G_2}(u_{2n})\cap \{u_{t+1},u_{t-1}\}\neq \emptyset$ by $|(N_{G_{2n}}(u_{2n+1})\cup N_{G_2}(u_{2n}))\cap V(C)|\geq n-1$. Without loss of generality, we assume $u_{2n}u_{t+1}\in E(G_2)$. By the similar argument to that in the proof of Lemma \ref{lem2-1} and $d^{-}_{D}(u_2)\geq \left \lfloor \frac{n}{2} \right \rfloor$, there exists $s\in\{3,\dots,2n-3\}$ such that $u_s\in N^{-}_{D}(u_2)$ and $u_{s+1}u_3\in E(G_t)$. Thus  $(C-\{u_2u_3,u_su_{s+1},u_tu_{t+1}\})\cup\{u_{2n+1}u_t,u_3u_{s+1},u_2u_s,u_{2n}u_{t+1},u_{2n-1}u_{2n}\}$ is a $\mathbf{G}$-transversal isomorphic to a Hamiltonian path with an associated injective $\phi_5$ that arises from $\phi$ by setting $\phi_5(u_{2n+1}u_t)=2n$, $\phi_5(u_3u_{s+1})=t$, $\phi_5(u_2u_s)=s$ and $\phi_5(u_{2n}u_{t+1})=2$, thus $(\mathrm{e_3})$ holds.

Combining above arguments, Claim 5.3 holds.
\end{proof}

\noindent\textbf{Claim 5.4.}
If $(\mathrm{e_2})$ holds, then $(\mathrm{e_3})$ holds.

\begin{proof}
Without loss of generality, let $C=u_1u_2\cdots u_{2n}u_1$ be a partial $\mathbf{G}$-transversal isomorphic to a cycle of order $2n$ with an associated injective $\sigma$: $E(C)\to [2n]$ such that $\sigma(u_iu_{i+1})=i$ for $i\in[2n-1]$ and $\sigma(u_{2n}u_1)=2n$. Let $D=(V,A)$ be a digraph on vertex set $V$ such that
\[A=\bigcup_{u_i\in X\setminus\{u_{2n+1}\}}\{\overrightarrow{u_iu}:u\in Y, u_iu\in E(G_i) \text{ and } u\neq u_{i+1}\}.\]
Then $|A|=\displaystyle\sum_{u\in Y}d^{-}_{D}(u)=\sum_{u\in X\setminus\{u_{2n+1}\}}d^{+}_{D}(u)\geq n\left \lceil \frac{n-1}{2}  \right \rceil$, thus there exists $u_t\in Y$ such that $d^{-}_{D}(u_t)\geq\left \lceil \frac{n-1}{2}  \right \rceil$ by $|Y|=n$.

Now we define the sets $J_7$ and $J_8$ as follows:
\[
J_7=\{j\in\{1,3,\dots,2n-1\}:u_j\in N^{-}_{D}(u_t)\},
\]
\[
J_8=\{j\in\{1,3,\dots,2n-1\}:u_{j+1}u_{2n+1}\in E(G_t)\}.
\]
Clearly, $|J_7|=d^{-}_{D}(u_t)\geq\left \lceil \frac{n-1}{2}  \right \rceil$ and
$|J_8|=d_{G_t}(u_{2n+1})\geq \left \lceil \frac{n+1}{2}  \right \rceil$. In addition, $t-1\notin J_7$ by $u_{t-1}\notin N^{-}_{D}(u_t)$.

If $t-1\in J_8$, then $u_tu_{2n+1}\in E(G_t)$, thus $u_{2n+1}u_tC^{-}u_{t+1}$ is a $\mathbf{G}$-transversal isomorphic to a Hamiltonian path, which implies $(\mathrm{e_3})$ holds.

If $t-1\notin J_8$, then $|J_7\cup J_8|\leq n-1$. Since $|J_7|+|J_8|\geq \left \lceil \frac{n-1}{2}  \right \rceil+\left \lceil \frac{n+1}{2}  \right \rceil\geq n$, we have $J_7\cap J_8\neq\emptyset$.
Thus there exists $j\in\{1,3,\dots,2n-1\}\setminus\{t-1\}$ such that $u_j\in N^{-}_{D}(u_t)$ and $u_{j+1}u_{2n+1}\in E(G_t)$. It follows that $u_{2n+1}u_{j+1}Cu_tu_jC^{-}u_{t+1}$ is a $\mathbf{G}$-transversal isomorphic to a Hamiltonian path with an associated injective $\sigma_1$ that arises from $\sigma$ by setting $\sigma_1(u_{j+1}u_{2n+1})=t$ and $\sigma_1(u_tu_j)=j$, thus $(\mathrm{e_3})$ holds.

Therefore, Claim 5.4 holds.
\end{proof}

By Claims 5.2-5.4, we complete the proof of Theorem \ref{thm1-5}.
{\hfill $\square$ \par}

\section*{Funding}
\hspace{1.5em}This work is supported by the National Natural Science Foundation of China (Grant
Nos. 12371347, 12271337).

\section*{Declarations}
\textbf{Conflict of interest} The authors declare that they have no known competing financial interests or personal relationships that could have appeared to influence the work reported in this paper.\\
\textbf{Data availability} No data was used for the research described in the article.

\end{document}